\documentclass[12pt,letterpaper,reqno]{amsart}

\usepackage{amssymb}
\usepackage[hidelinks]{hyperref}
\usepackage[T1]{fontenc}

\usepackage{txfonts}  

\newcommand{\N}{\mathbb N}

\newcommand{\R}{\mathbb R}
\newcommand{\C}{\mathbb C}
\newcommand{\T}{\mathbb T}

\begin{document}

\title{Holomorphic extendability in $\mathbf C^n$ as a rare phenomenon}
\author{Nikolaos Georgakopoulos}
\address{University of Athens, Department of Mathematics, 157 84 Panepistemiopolis, Athens, Greece}
\email{nikolaosgmath@gmail.com}

\theoremstyle{plain}

\newtheorem{prop}{Proposition}
\newtheorem{thm}[prop]{Theorem}
\newtheorem{cor}[prop]{Corollary}
\newtheorem{lem}[prop]{Lemma}

\theoremstyle{definition}
\newtheorem{defn}[prop]{Definition}
\newtheorem{rem}[prop]{Remark}
\newtheorem{que}[prop]{Question}
\newtheorem{nota}[prop]{Notation}

\begin{abstract} We consider various notions of holomorphic extendability of complex valued functions defined on subsets of $\C^n$, including one-sided extendability. We show that in the relevant function spaces, these 
phenomena of holomorphic extendability are rare in the topological sense, generalizing several results of \cite{BigPaper} in dimensions $n\ge 2$. \end{abstract}

\maketitle

\tableofcontents

\section{Introduction}

The authors of \cite{BigPaper} establish that in various spaces of functions defined on subsets of $\C$,  nowhere holomorphic extendability is a generic phenomenon. Starting with a domain $\Omega\subseteq \C$ defined by a finite number of pairwise disjoint Jordan curves, they prove that it is rare, in the topological sense, for a holomorphic function $f:\Omega\to \C$ to have a holomorphic extension at some boundary point $z\in \partial \Omega$. A holomorphic extension in this setting is a holomorphic function $F$ defined on an open neighborhood $U$ of $z$ such that $F=f$ over a connected component of $\Omega\cap U$. Using Baire's Theorem, they prove that the collection of functions $f$ with the previous property is a meager ($F_{\sigma}$ with empty interior) subset of the space $A^p(\Omega)$ for any  $p\in \{0,1,...\}\cup \{+\infty\}$. The space $A^p(\Omega)$ consists of all holomorphic functions $f:\Omega\to \C$ whose derivatives $f^{(l)}$, $l\in \N, l\le p$, have continuous extensions over the boundary $\partial \Omega$; its topology is induced by the seminorms $\sup_{z\in \Omega, |z|\le N}|f^{(l)}(z)|$ for each $N\in \N$ and $l\in \N$, $l\le p$.

Next, they prove the analogous result for the spaces $C^p(\T)$ of $C^p$ smooth functions $\T\to \C$, $\T$ being the unit circle in $\C$. Here, $p\in \{0,1,...\}\cup \{+\infty\}$ and extendability of $f\in C^p(\T)$ at $z\in \T$ amounts to the existence of a holomorphic $F$ defined on an open neighborhood $U\subseteq \C$ of $z$ so that $F=f$ over $\T\cap U$. Furthermore, they show that real analyticity of $f:\T\to \C$ at a point $e^{i\theta_0}$, that is the existence of a power series of a real variable $\theta$ converging to $f(e^{i\theta})$ for $\theta$ near $\theta_0$, is equivalent to holomorphic extendability of $f$ at $e^{i\theta_0}$. This implies that real analyticity is also a rare phenomenon in $C^p(\T)$. The results regarding the circle can be generalized to boundaries of Jordan domains, once the original argument is refined to prove the following: If $L\subseteq \C$ is a perfect set then the nowhere holomorphically extendable continuous functions $f:L\to \C$ form a dense $G_{\delta}$ subset of $C(L)$.

Finally, the authors of \cite{BigPaper} show that holomorphic extendability from one side is also a rare phenomenon, and proceed to define and study the $p$-continuous analytic capacities. These are related to the following problem: Fix a $p\in \{0,1,...\}\cup \{+\infty\}$, a domain $\Omega\subseteq \C$ and a compact set $L\subseteq \C$; then we obviously have $A^p(\Omega)\subseteq A^p(\Omega\setminus L)$ and the problem is to determine when $A^p(\Omega)=A^p(\Omega\setminus L)$. For $p=0$ the solution is given by the continuous analytic capacities: a compact set $L\subseteq \C$ has the previous property for every domain $\Omega$ containing it, if and only if its continuous analytic capacity is zero: $\alpha(L)=0$ (\cite{Garnett}). In \cite{BigPaper}, the authors prove that for any $p\in \{0,1,...\}\cup \{+\infty\}$ and compact $L$, $A^p(\Omega)=A^p(\Omega\setminus L)$ for all domains $\Omega\supseteq L$ if and only if $\alpha_p(L)=0$, $\alpha_p$ denoting the $p$-continuous analytic capacity that they define. They go on to show that these capacities are distinct, in the sense that there exist compact sets $L$ for which $\alpha_0(L)=0$ but $\alpha_1(L)\neq 0$.

In this article, we manage to extend several of the previous results in higher dimensions, where subsets of $\C^n$, $n\ge 1$, are used. At the outset, we are presented with a fundamental difficulty: Some of those results do not hold true in dimensions $n>1$. Indeed, Hartog's phenomenon (\cite{Gauthier}) shows the existence of open and connected sets $\Omega\subseteq \C^n$ where all holomorphic functions $f:\Omega\to \C$ have holomorphic extensions over some points in $\partial \Omega$. Hartog's figure is an elementary example showcasing this behavior (\cite{Gauthier}). Thus, it is apparent that we can not hope for a very straightforward extension of the one dimensional results in several dimensions; we will have to impose additional conditions on the subsets of $\C^n$ that we are using.

First of all, in section \ref{Extendability in domains}, we are concerned with extendability at the boundary, of functions defined on open and connected sets $\Omega\subseteq \C^n$. A Theorem in \cite{NestorAndDaras} can be used to show that if $\Omega$ is convex, then the collection of non-holomorphically extendable functions in $A^{\infty}(\Omega)$ contains a dense $G_{\delta}$ of $A^{\infty}(\Omega)$. In Theorem \ref{HomolomorphicExtendabilityAp} however, we prove that the aforementioned collection is actually a $G_{\delta}$ itself, for a more general class of sets $\Omega$. The condition we impose on $\Omega$ is related to being able to separate $\Omega$ from any point in $\overline{\Omega}^c$ via an $n-1$ dimensional complex hyperplane. This is quite more general than the convex case, and also trivially contains the one-dimensional case. In the rest of section \ref{Extendability in domains}, we examine for which $\Omega$ and compact sets $L\subseteq \Omega$, the equality $A^p(\Omega)=A^p(\Omega\setminus L)$ is true. If $\Omega\setminus L$ is connected then the answer, for $n\ge 2$, is affirmative by Hartog's Extension Theorem. But even if it is not connected, we always have a dichotomy: either $A^p(\Omega)=A^p(\Omega\setminus L)$ or $A^p(\Omega\setminus L)\setminus A^p(\Omega)$ is a dense $G_{\delta}$ of $A^p(\Omega\setminus L)$ (Theorem \ref{Dichotomy}). In other words, we can either extend all functions from $\Omega\setminus L$ to $\Omega$, or the functions that can be extended are topologically very few, constituting a meager set. We do not delve in a deeper classification, by defining analytic capacities in several variables, and we do not see if such an investigation is possible.

In section \ref{Real Analyticity on the Polydisc}, we exclusively treat the case of the $n$ dimensional open polydisc $D^n$, defining notions such as real analyticity and holomorphic extendability for functions defined on $\T^n$. For functions in $A^p(D^n)$ there are then two notions of holomorphic extendability, but we promptly establish their equivalence (Proposition \ref{ExtendabilityEquivalentDefinitions}). Real analyticity and holomorphic extendability are also shown to be equivalent, and from this it follows that they are topologically rare phenomena in $A^p(D^n)$ (Theorem \ref{RealAnalyticityinApDn}).

In section \ref{Extendability of functions defined on closed sets}, we consider extendability of functions defined on closed sets of $\C^n$, rather than domains. We impose two conditions on these closed sets:  The first is related to separating them from points in their complement using $n-1$ dimensional complex hyperplanes, as in section \ref{Extendability in domains}. The other condition ensures that we have limit points in every direction, enabling us to use the principle of analytic continuation in each complex variable separately. We can then prove Theorem \ref{TheoremWithClosedSet}: If $L$ is has the two aforementioned conditions, then the collection of non-holomorphically extendable functions $L\to \C$ is a dense $G_{\delta}$ subset of $C(L)$. This immediately implies that non-extendability is a generic phenomenon in the spaces of functions $C^p(\T^n)$, $C(\partial D^n)$ and $C(\partial B^n)$, where $B^n$ is the unit ball of $\C^n$ (we note that the distinguished boundary $\T^n$ of $D^n$ is a proper subset of the boundary $\partial D^n$ of $D^n$). Finally, we can also generalize the results regarding $\T^n$ to products of Jordan domains.

In the final section, we examine the case of one-sided holomorphic extendability. There are some additional technicalities to be considered and for this reason, we define three notions of one-sided extendability. The regular notion involves extensions by functions that are continuous on the boundary, but we also consider another notion where only separate continuity is assumed, and another where we assume Lipschitz continuity. We are able to show that for sets like the polydisc, the ball, the distinguished boundary of the polydisc and products of Jordan domains, the nowhere one-sided-extendable functions, in either the separate continuous or the Lipschitz continuous sense, form a dense $G_{\delta}$ subset of the relevant function spaces. The case of usual continuity is ``sandwiched'' between the other two, and thus the collection of nowhere one-sided-extendable functions in the usual sense both contains and is contained in a dense $G_{\delta}$. Therefore, it is a residual set, but it remains to be determined if it actually is a $G_{\delta}$ set.

\section{Extendability in domains of $\C^n$}\label{Extendability in domains}

Let us first fix some notation. The natural number $n$ always denotes the (complex) dimension of the ambient space $\C^n$. The set $D(z,r)$ denotes the open disc in $\C$ with center $z\in \C$ and radius $r>0$. The open unit disc is denoted by $D$ and the unit circle by $\T$. The set $B(z,r)$ denotes the open ball in $\C^n$ centered at $z\in \C^n$ with radius $r$; $B^n=B(0,1)$ is the unit ball in $\C^n$. We utilize the usual multi-index notation, where $\alpha=(a_1,...,a_n)\in \N^n$ is a multi-index, $|\alpha|=a_1+\cdots+a_n$, $\alpha!=a_1!...a_n!$, $z^{\alpha}=z_1^{a_1}\cdots z_n^{a_n}$ for $z=(z_1,...,z_n)$ and $\partial^a/\partial z^{\alpha}=\partial^{|\alpha|}/\partial z_1^{a_1}\cdots \partial z_n^{a_n}$.

Let $\Omega\subseteq \C^n$ be open and connected, and $p\in \{0,1,...\}\cup\{+\infty\}$. The space $A^p(\Omega)$ consists of all holomorphic functions $f$ on $\Omega$, whose partial derivatives 
\begin{equation}\frac{\partial^{\alpha}}{\partial z^{\alpha}}f\end{equation}
extend continuously over $\overline{\Omega}$, for every multi-index $\alpha\in \N^n$ with $|\alpha|\le p$. Its topology is given by the semi-norms
\begin{equation*}|f|_{\alpha,N}=\sup_{z\in \overline{\Omega}, \|z\|\le N}\left|\frac{\partial^{\alpha}}{\partial z^{\alpha}}f(z)\right|\text{ , }\alpha\in \N^n\text{ , }|\alpha|\le p\text{ , }N\in \N\end{equation*}
and it is a Fr\'echet space for all $p$. If $p<+\infty$ and $\Omega$ is bounded then it is a Banach space. For $p=0$ we denote $A(\Omega)=A^0(\Omega)$.

If $\{\alpha_j\}_{j\in \N}$ is an enumeration of $\N^n$ via multi-indices, we denote
\begin{equation}S_N(f,\zeta)(z)=\sum_{j=0}^N\frac{\partial^{\alpha_j}f}{\partial z^{\alpha_j}}\frac{(z-\zeta)^{\alpha_j}}{\alpha_j!}\end{equation}
for $\zeta,z\in \Omega$, $N\in \N$, $f$ holomorphic on $\Omega$.

\begin{thm}\label{NestorAndDaras}\cite{NestorAndDaras} If $\Omega$ is convex, then there is an $f\in A^{\infty}(\Omega)$ such that for every convex and compact set $K\subseteq \overline{\Omega}^c$, and every polynomial $h$ in $n$ complex variables, there is a sequence $N_k\in \{0,1,...\}$ such that $S_{N_k}(f,\zeta)(z)\to h(z)$ uniformly for $(\zeta,z)\in L\times K$ for every compact $L\subseteq \overline{\Omega}$. The collection of all such $f$ is a dense $G_{\delta}$ in  $A^{\infty}(\Omega)$.\end{thm}

A holomorphic function $f$ on $\Omega$ is holomorphically extendable at $z\in \partial \Omega$ if there is an open set $U\subseteq \C^n$ containing $z$ and a holomorphic function $F$ on $U$ such that $F=f$ on a connected component of $U\cap \Omega$. The function $f$ is non-holomorphically extendable if it is not holomorphically extendable at any point of the boundary. Note that if $f$ is non-holomorphically extendable at countably many points $z_k$ and the sequence $z_k$ is dense in $\partial \Omega$, then $f$ is non-holomorphically extendable. 

If $f$ is holomorphic over an open polydisc centered at $\zeta$, then the Taylor expansion of $f$ around $\zeta$ converges absolutely and uniformly to $f$ over the compact subsets of that  open polydisc. If $f\in A^{\infty}(\Omega)$ satisfies the conclusion of Theorem \ref{NestorAndDaras}, then $f$ is non-holomorphically extendable. Therefore,

\begin{thm}\label{WeakTheoremNestorDalas}If $\Omega\subseteq \C^n$ is open and convex, then the collection of non-holomorphically extendable functions in $A^{\infty}(\Omega)$ is residual: it contains a dense $G_{\delta}$ subset.\end{thm}

We can in fact prove that this collection is a $G_{\delta}$ set, and this applies to more general sets $\Omega$, that satisfy the following condition:

\begin{itemize}\item[*] $\Omega\subseteq \C^n$ is open and connected and for every point $\zeta\in \partial \Omega$ there is a point $w\in \overline{\Omega}^c$ arbitrarily close to $\zeta$ and a non zero vector $v\in \C^n$ so that the $n-1$ dimensional complex hyperplane $H=w+\{v\}^{\perp}=\{z\in \C^n:\langle z,v\rangle=\langle w,v\rangle\}$ in $\C^n$ is disjoint from $\overline{\Omega}$. Here, $\langle \cdot,\cdot\rangle$ denotes of course the usual complex inner product.
\end{itemize}
This condition will be referred to as the ``star condition'' from now on. If $\Omega$ satisfies it, then we must have $\partial \Omega=\partial \overline{\Omega}$, which is also equivalent to $\Omega=\left(\overline{\Omega}\right)^{\circ}$. In dimension $n=1$, every domain $\Omega\subseteq \C$ such that $\partial \Omega=\partial \overline{\Omega}$, satisfies the star condition, as we can use $H=\{w\}$. In higher dimensions, we can prove the following:

\begin{lem}Every open and convex $\Omega\subseteq \C^n$ satisfies the star condition.\end{lem}

\begin{proof}First, it is well known that $\Omega=\left(\overline{\Omega}\right)^{\circ}$ for open and convex $\Omega$. It thus suffices to construct for every $w\in \overline{\Omega}^c$ an $n-1$ dimensional complex hyperplane $H$ that contains $w$ and is disjoint from $C=\overline{\Omega}$. After a translation we may assume $w=0$; if $v$ is the nearest point of $C$ to the origin, we set $H=\{v\}^{\perp}$. If there is some $z\in H\cap C$ then 
for any $t\in (0,1)$, we have $(1-t)z+tv\in C$ and $z\perp v$ hence by the Pythagorean Theorem,
\begin{equation*}\|v\|^2\le \|(1-t)z+tv\|^2=(1-t)^2\|z\|^2+t^2\|v\|^2\implies \sqrt{\frac{1+t}{1-t}}\|v\|\le \|z\|\end{equation*}
Since $0\notin C$, we have that $v\neq 0$ so sending $t\to 1$ gives that $\|z\|=+\infty$, contradiction. Therefore, $H\cap C=\emptyset$ and the proof is complete.\end{proof}

The converse to this is clearly not true: In dimension $1$, every domain $\Omega$ such that $\partial\Omega=\partial\overline{\Omega}$ satisfies the star condition but is not necessarily convex. In dimension $2$, the set $\Omega=\{(z,w)\in \C^2:1<|z|<2\}$ is not convex, but for any point $(z_1,z_2)\notin \overline{\Omega}$, the line $z=z_1$ passes through $(z_1,z_2)$ and does not intersect $\overline{\Omega}$ (and clearly, $\partial \Omega=\partial \overline{\Omega}$).

\begin{thm}\label{HomolomorphicExtendabilityAp}If $p\in \{0,1...\}\cup \{+\infty\}$ and $\Omega\subseteq \C^n$  satisfies the star condition then
\begin{itemize}\item[1.] For any $\zeta\in \partial \Omega$, the collection of functions in $A^p(\Omega)$ that are not holomorphically extendable at $\zeta$ is a dense $G_{\delta}$ subset of $A^p(\Omega)$.\smallbreak
\item[2.] The collection of non-holomorphically extendable functions in $A^p(\Omega)$ is a dense $G_{\delta}$ subset of $A^p(\Omega)$.\end{itemize}
In particular, items 1. and 2. hold for open and convex $\Omega$.\end{thm}

\begin{proof}The second item follows immediately from the first, by utilizing a dense countable sequence of points $\zeta\in \partial \Omega$ where $f$ is not holomorphically extendable, in conjunction with Baire's Theorem. So we shall prove the first statement.

Let $M,r\in (0,+\infty)$ and $V$ be a connected component of $B(\zeta,r)\cap \Omega$. We define $S=S_{M,r,V}$ to be the collection of all $f\in A^p(\Omega)$ for which there is a holomorphic $F$ on $B(\zeta,r)$ bounded by $M$ and agreeing with $f$ over $V$. We will show that $S$ is closed with empty interior in $A^p(\Omega)$.

To show that $S$ is closed, let $f_m\in S$ converge to $f$ in the topology of $A^p(\Omega)$. There are holomorphic $F_m$ on $B(\zeta,r)$ agreeing with $f_m$ on $V$ and bounded by $M$. By Montel's Theorem (\cite{Gauthier}), there is a subsequence $F_{k_m}$ and a holomorphic function $F$ on $B(\zeta,r)$ such that $F_{k_m}\to F$ uniformly over all compact subsets of $B(\zeta,r)$. Since $F_{k_m}=f_{k_m}$ over $V$, it follows that $F=f$ over $V$, hence $f\in S$ ($F$ is bounded by $M$ as the point-wise limit of functions $F_{k_m}$ bounded by $M$).

To show that $S$ has empty interior, let us assume on the contrary that it has an interior point $f$. Then there is some $\epsilon>0$ and $l,N\in \N,l\le p$, so that if $g\in A^p(\Omega)$ satisfies
\begin{equation}\label{DistanceOfFunctions1}\sup_{z\in \Omega, \|z\|\le N}\left|\frac{\partial^{\alpha}f}{\partial z^{\alpha}}(z)-\frac{\partial^{\alpha}g}{\partial z^{\alpha}}(z)\right|<\epsilon\text{ , }\forall \alpha\in \N^n\text{ , }|\alpha|\le l\end{equation}
then $g\in S$. By the star condition, there is a $w\in B(\zeta,r)\setminus \overline{\Omega}$ and an $n-1$ dimensional complex hyperplane containing $w$ but not intersecting $\overline{\Omega}$: there are $v_i\in \C$ not all zero, such that $H=\{(z_1,...,z_n)\in \C^n:(z_1-w_1)v_1+\cdots+(z_n-w_n)v_n=0\}$ does not intersect $\overline{\Omega}$. Consider $h:H^c\to \C$,
\begin{equation}h(z)=\frac{\delta}{(z_1-w_1)v_1+\cdots+(z_n-w_n)v_n}\end{equation}
for small $\delta>0$. If $g=f+h$, then clearly $g\in A^p(\Omega)$, and for small enough $\delta$, $\eqref{DistanceOfFunctions1}$ is satisfied, hence $g\in S$. If $F,G$ are holomorphic on $B(\zeta,r)$ and agree with $f,g$ over $V$, then $G-F=h$ over $V$. The set $B(\zeta,r)\setminus H$ is connected, which can be seen by counting real dimensions: $B(\zeta,r)$ is the unit ball in $\R^{2n}$ so $2n$ dimensional, while $H$ is only $2n-2$ dimensional. Therefore, since $G-F=h$ over an open subset of $B(\zeta,r)\setminus H$, it follows that $G-F=h$ over $B(\zeta,r)\setminus H$. But because $G,F$ are bounded by $M$, this implies that $h$ is bounded over $B(\zeta,r)\setminus H$, contradicting that $h(z)$ is unbounded as $z\to w$, $z\notin H$.

So far, we have proven that $S_{M,r,V}$ is closed with empty interior in $A^p(\Omega)$. The set of functions in $A^p(\Omega)$ that are holomorphically extendable at $\zeta$ is $A=\cup_{M>0, r>0, V}S_{M,r}$, $V$ ranging over the countably many connected components of $B(\zeta,r)\cap \Omega$. We may write this as a countable union, using $M=1/m, r=1/k$, $m,k\in \N$ and by enumerating the components $V$; in this way, the complement $A^c$ is the countable intersection of dense open sets in the complete metric space $A^p(\Omega)$. By Baire's Theorem, $A^c$ is a dense $G_{\delta}$ subset of $A^p(\Omega)$. The set $A^c$ is the collection of functions that are not holomorphically extendable at $\zeta$.\end{proof}

In dimension $n=1$, every domain $\Omega$ such that $\Omega=\left(\overline{\Omega}\right)^{\circ}$ satisfies the hypothesis of this Theorem hence its conclusion as well. This generalizes Theorem 5.12 of \cite{BigPaper} which is stated only for bounded domains $\Omega\subseteq \C$ defined by a finite number of pairwise disjoint Jordan curves. In dimensions $n\ge 2$, there are domains $\Omega$ such as Hartog's figure for which the conclusion of Theorem \ref{HomolomorphicExtendabilityAp} is not true.

\begin{rem}We also note that the previous proof works with a different definition of holomorphic extendability. In this alternate definition, we require the holomorphic extension $F$ to agree with the original function $f$ over the entire set $B(\zeta,r)\cap \Omega$ and not just a connected component. The proof that the functions in $A^p(\Omega)$ that are not holomorphically extendable in this sense at a point (or at all points) of $\partial \Omega$, is similar to the proof above and simpler.

This alternate definition is equivalent to the one given before for convex $\Omega$ (the intersection of convex sets is convex hence connected). They are not equivalent however, if $\Omega$ only has the star condition, even in dimension $n=1$. Indeed, the set $\Omega=\cup_{m=1}^{+\infty}\{z\in \C:1/(m+1)<\Re z<1/m\}\cup \{z\in \C: Im z>1\}\cup \{z\in \C: \Re z<0\}$ satisfies the star condition, but $0\in \partial \Omega$ and $B(0,1/2)\cap \Omega$ is disconnected.
\end{rem}

Now let $\Omega$ be an open subset of $\C^n$, $L$ a compact subset of $\Omega$ and $G=\Omega\setminus L$. A function $f\in A^p(G)$ may have an extension in $A^p(\Omega)$ (that is necessarily unique) or it may not. 

If $G$ is connected and $n\ge 2$, then Hartog's Extension Theorem applies, so $f$ has a holomorphic extension over $\Omega$. The function $f\in A^p(G)$ has by definition a $C^p$ smooth extension over $\overline G$, which contains $\partial \Omega$ because $L$ is compact and $\Omega$ is open, hence a holomorphic extension of $f$ over $\Omega$ will also be in $A^p(\Omega)$. We conclude that if $G$ is connected and $n>1$ then every function in $A^p(G)$ has a (unique) extension in $A^p(\Omega)$. 

If $G$ is not connected however, we can construct counterexamples. For instance, we may have $\Omega$ be the unit ball in $\C^n$ and $L=\{z\in \C^n:1/4\le \|z\|\le 3/4\}$; $G=\{z\in \C^n:\|z\|<1/4\}\cup \{z\in \C^n:3/4<\|z\|<1\}$ and the function that is $0$ on the former set of the union and $1$ in the latter is in $A^p(G)$ but has no extension in $A^p(\Omega)$. In this case, we can prove that this is generically true for any function in $A^p(G)$:

\begin{thm}\label{Dichotomy} Let $\Omega$ be an open subset of $\C^n$, $L$ a compact subset of $\Omega$ and $G=\Omega\setminus L$. If there is a function $f_0\in A^p(G)$ with no extension in $A^p(\Omega)$, then the set of all such functions is a dense $G_{\delta}$ subset of $A^p(G)$.\end{thm}

\begin{proof}Let $0<M<\infty$ and $S=S_M$ be the set of functions $f\in A^p(G)$ for which there is an $F\in A^p(\Omega)$ bounded by $M$ such that $F=f$ on $G$.

$S$ is closed in $A^p(G)$: If $f_m\in S$ have extensions $F_m\in A^p(\Omega)$ and $f_m\to f$ in the topology of $A^p(G)$, then by Montel's Theorem (\cite{Gauthier}), there is a subsequence $F_{k_m}\in A^p(\Omega)$ with $F_{k_m}\to F$ over the compact subsets of $\Omega$, for some $F$ holomorphic on $\Omega$. We have $F=f$ over $G$ and $F$ is bounded by $M$. Finally, $\partial \Omega\subseteq \partial G$, and as $F=f\in A^p(G)$, it follows that $F\in A^p(\Omega)$.

$S$ has empty interior: If on the contrary there is some $f\in S^{\circ}$, then there are $l,N\in \N$, $l\le p$ and $\epsilon>0$ so that whenever some $g\in A^p(G)$ satisfies
\begin{equation}\label{DistanceOfFunctions}\sup_{z\in \overline G,\|z\|\le N}\left|\frac{\partial^{\alpha}f}{\partial z^{\alpha}}(z)-\frac{\partial^{\alpha}g}{\partial z^{\alpha}}(z)\right|<\epsilon\text{ , }\forall \alpha\in \N^n\text{ , }|\alpha|\le l\end{equation}
then $g\in S$. The function $h=\delta f_0+f$ for suitably small $\delta>0$ satisfies this condition, hence $h\in S$. If $F,H\in A^p(\Omega)$ are the extensions of $f,h$ over $\Omega$, then $(H-F)/\delta$ is a $A^p(\Omega)$ extension of $f_0$, contradicting our assumption on $f_0$.

Finally, the set of functions $f\in A^p(G)$ with no extension in $A^p(\Omega)$ is the countable intersection $\cap_{M\in \N}S_M^c$; the sets $S_M^c$ are open and dense hence Baire's Theorem allows us to conclude that the intersection is a dense $G_{\delta}$ of $A^p(G)$.
\end{proof}

We thus have established a dichotomy: Either every function in $A^p(G)$ will have an extension in $A^p(\Omega)$, or the collection of functions in $A^p(G)$ without such an extension will be a dense $G_{\delta}$ in $A^p(\Omega)$. Both possibilities can of course occur.

\section{Real analyticity on the polydisc}\label{Real Analyticity on the Polydisc}

Consider an open polydisc in $\C^n$; for simplicity let us use $D^n$, $D$ being the open unit disc in $\C$. The distinguished boundary of $D^n$ is $\T^n$, $\T$ being the unit circle in $\C$.

\begin{prop}\label{RealAndComplexAnalyticity}Let $(\theta_1,...,\theta_n)\in \R^n$ and $f:\T^n\to \C$. The following are equivalent
\begin{itemize}\item[1.] There is a power series
\begin{equation}I(t_1,...,t_n)=\sum_{k_1,...,k_n\ge 0}a_{k_1,...,k_n}(t_1-\theta_1)^{k_1}\cdots (t_n-\theta_n)^{k_n}\end{equation}
converging absolutely and uniformly over $(t_1,...,t_n)\in \prod_{i=1}^n[\theta_i-\delta,\theta_i+\delta]$ for some $\delta>0$, so that $I(t_1,...,t_n)=f(e^{it_1},...,e^{it_n})$ for $(t_1,...,t_n)\in \prod_{i=1}^n[\theta_i-\delta,\theta_i+\delta]$.
\item[2.] There is a power series
\begin{equation}J(z_1,...,z_n)=\sum_{k_1,...,k_n\ge 0}b_{k_1,...,k_n}(z_1-e^{i\theta_1})^{k_1}\cdots (z_n-e^{i\theta_n})^{k_n}\end{equation}
converging absolutely and uniformly over $(z_1,...,z_n)\in \prod_{i=1}^n\overline{ D(\theta_i,\delta)}$ for some $\delta>0$, so that $J(e^{it_1},...,e^{it_n})=f(e^{it_1},...,e^{it_n})$ for $(t_1,...,t_n)\in \prod_{i=1}^n[\theta_i-\delta,\theta_i+\delta]$.\end{itemize}
\end{prop}

\begin{proof}Shrinking $\delta$ if necessary, we may use open intervals $(\theta_i-\delta,\theta_i+\delta)$ and discs $D(\theta_i,\delta)$ instead of closed ones.

Let $\gamma:\R^n\to \T^n$, $\gamma(t_1,...,t_n)=(e^{it_1},...,e^{it_n})$. If $A=\prod_{i=1}^n(\theta_i-\delta,\theta_i+\delta)$ for $\delta>0$, then $\gamma:A\to B=\gamma(A)$ is a diffeomorphism for small enough $\delta$. The map $\gamma$ can be extended to a holomorphic map $\Gamma:\C^n\to \C^n$, $\Gamma(z_1,...,z_n)=(e^{iz_1},...,e^{iz_n})$; if we restrict $\Gamma$ to a sufficiently small open neighborhood $U$ of $A\subseteq \C^n$, then $\Gamma:U\to V=\Gamma(U)$ becomes a biholomorphism extending $\gamma:A\to B$.

If $f$ has the property in the first item then $I=f\circ \gamma$ on $A$. The power series $I$ can be extended to a complex power series $I(z_1,...,z_n)$ converging absolutely and uniformly for $z_i\in D(\theta_i,\delta)$. The extended power series $I$ is holomorphic on $\prod_{i=1}^nD(\theta_i,\delta)$ and we may take $U=\prod_{i=1}^nD(\theta_i,\delta)$ by shrinking $\delta$ if needed. Therefore, $I\circ \Gamma^{-1}:V\to \C$ is holomorphic  hence it has an absolutely and uniformly convergent power series $J$ over a polydisc in $V$ centered at $(e^{i\theta_1},...,e^{i\theta_n})$. Over $B$, $J=I\circ \Gamma^{-1}=I\circ \gamma^{-1}=f$.

Conversely, if $f$ has the property in the second item, then $J\circ \Gamma$ is holomorphic, with a power series $I$ converging absolutely and uniformly to it over a polydisc in $U$ centered at $(t_1,...,t_n)$. Then over $A$, $I=J\circ \Gamma=f\circ \gamma$ as desired.
\end{proof}

A function $f:\T^n\to \C$ is real analytic at $(e^{i\theta_1},...,e^{i\theta_n})$ if it satisfies one (hence both) of the conditions in Proposition \ref{RealAndComplexAnalyticity}.

A function $f:\T^n\to \C$ is holomorphically extendable at $(e^{i\theta_1},...,e^{i\theta_n})$ if there is an open neighborhood $U$ of this point and a holomorphic function $F:U\to \C$ agreeing with $f$ over $U\cap \T^n$. By Proposition \ref{RealAndComplexAnalyticity}, we derive the following: 

\begin{prop}\label{RealAnalyticIsHolomorphicallyExtendable}A function $f:\T^n\to \C$ is holomorphically extendable at a point of $\T^n$ if and only if it is real analytic at that point.\end{prop}

We have already defined a different notion of holomorphic extendability for functions $f\in A^p(D^n)$ at points of the boundary, in section 2. But we can show that the two definitions are equivalent:

\begin{prop}\label{ExtendabilityEquivalentDefinitions}Let $p\in \{0,1,...\}\cup \{+\infty\}$, $f\in A^p(D^n)$ and $\zeta\in \T^n$. The function $f$ is holomorphically extendable at $\zeta$ if and only if $f|_{\T^n}$ is holomorphically extendable at $\zeta$.\end{prop}

\begin{proof}Clearly, if $f\in A^p(D^n)$ is holomorphically extendable at $\zeta$ then $f|_{\T^n}$ is as well. The converse is less obvious. Assuming that $f|_{\T^n}$ is holomorphically extendable at $\zeta$, we obtain a holomorphic function $F$ defined on a small open polydisc $U=D_1\times...\times D_n$ centered at $\zeta$, so that $F=f$ over $\T^n\cap U$. It suffices to prove that the function $G:U\to \C$ defined as $G=f$ on $D^n\cap U$ and $G=F$ on $U\setminus D^n$, is holomorphic. We shall prove this by inducting on $n$.

If $n=1$, $G$ is continuous on $D\cup U$ and holomorphic everywhere except for the arc $\T\cap U$; it follows by elementary complex analysis (e.g. by Morrera's Theorem) that $F$ is holomorphic on $D\cup U$. 

Now let $n\ge 2$ and assume the result holds for $n-1$. By Hartog's Theorem, we only need to show that $G$ is separately holomorphic in each variable. We shall do it for the first variable, as the case of the other variables is entirely similar. So fix $(w_1,...,w_n)\in U$ and consider the function $g:D_1\to \C$, $g(z)=G(z,w_2,...,w_n)$.

If all $w_i$ have modulus $|w_i|<1$, then $g(z)=f(z,w_2,...,w_n)$ is holomorphic near $w_1$ by definition. So we may assume that $|w_j|\ge 1$ for some $j$. If we can pick $j\neq 1$ then $g(z)=F(z,w_2,...,w_n)$ is again holomorphic near $w_1$. So let us assume that $|w_i|<1$ for all $i\neq 1$, but $|w_1|\ge 1$. If $|w_1|>1$ then near $w_1$, $g(z)=F(z,w_2,...,w_n)$ so this case is also dealt with.

 It remains to consider the event where $|w_1|=1$ and $|w_i|<1$ for all $i\neq 1$. Then $g:D_1\to \C$ is equal to $f(z,w_2,...,w_n)$ for $|z|<1$ and to $F(z,w_2,...,w_n)$ for $|z|\ge 1$. Thus $g$ is holomorphic on $D_1\setminus \T$, and it remains to show that it is continuous on $D_1$, as this would imply holomorphy over $D_1$ by the one dimensional case. Continuity of $g$ reduces to $f(z,w_2,...,w_n)=F(z,w_2,...,w_n)$ for all $|z|=1$, $z\in D_1$. To prove this, fix such $z$ and consider the functions $\kappa:D^{n-1}\to \C$, $\kappa(z_2,...,z_n)=f(z,z_2,...,z_n)$ and $\lambda:U'\to \C$, $\lambda(z_2,...,z_n)=F(z,z_2,...,z_n)$ where $U'=D_2\times...\times D_n$. These are two holomorphic functions that agree on $T^{n-1}\cap U'$, hence by the induction hypothesis, we have that $\kappa=\lambda$ on their common domain of definition $D^{n-1}\cap U'$. We conclude that $f(z,w_2,...,w_n)=F(z,w_2,...,w_n)$ for all $z\in \T\cap D_1$, as desired. This proves that $g$ is holomorphic near $w_1$, hence that $G$ is holomorphic with respect to its first variable on $U$ (as $(w_1,...w_n)\in U$ was chosen arbitrarily). The proof is complete.
\end{proof}

Combining Proposition \ref{ExtendabilityEquivalentDefinitions} with Theorem \ref{HomolomorphicExtendabilityAp} and Proposition \ref{RealAnalyticIsHolomorphicallyExtendable}, we obtain: 

\begin{thm}\label{RealAnalyticityinApDn}For $p\in \{0,1,...\}\cup \{+\infty\}$, the collection of functions in $A^p(D^n)$ that are not real analytic at a fixed point (or at all points) of $\T^n$ is a dense $G_{\delta}$ of $A^p(D^n)$.\end{thm}

\section{Extendability of functions defined on closed sets}\label{Extendability of functions defined on closed sets}

Let $L$ be a closed subset of $\C^n$. The space $C(L)$ of continuous functions $L\to \C$ is a Fr\'echet space with seminorms $|f|_N=\sup_{z\in L, \|z\|\le N}|f(z)|$, $N\in \N$. If $L$ is compact, then $C(L)$ is a Banach space.

 A function $f\in C(L)$ is holomorphically extendable at a point of $L$ if there is a neighborhood $U$ of this point and a holomorphic function $F:U\to \C$ such that $F=f$ on $U\cap L$.

\begin{thm}\label{TheoremWithClosedSet}Let $L$ be a closed subset of $\C^n$ with the following two properties
\begin{itemize}
\item For every $\zeta\in L$ there is an $r_0>0$ such that either $B(\zeta,r_0)\subseteq L$ or for all $r\in (0,r_0)$ there is a point $w\in B(\zeta,r)\setminus L$ and an $n-1$ dimensional complex hyperplane $H$ containing $w$ and disjoint from $B(\zeta,r)\cap L$.
\item For any product $U=U_1\times...\times U_n$ of domains $U_i\subseteq \C$ and for any point $ (z_1,...,z_n)\in U\cap L$, the sets $\{w\in U_i: (z_1,...,z_{i-1},w,z_i,...,z_n)\in L\}$ have a limit point in $U_i$ for all $i=1,2,...,n$. 
\end{itemize}
Then the set of functions in $C(L)$ that are not holomorphically extendable at a given point of $L$ (or at all points of $L$) is a dense $G_{\delta}$ of $C(L)$. \\
In particular, in dimension $n=1$, a sufficient condition is that $L$ be perfect (closed with no isolated points).\end{thm}

\begin{proof}It suffices to prove that non-extendability at a fixed point $\zeta\in L$ is a generic phenomenon (we can then use a dense sequence of points $z_k\in L$ in conjunction with Baire's Theorem to prove the analogous result for non-extendability at all points of $L$). 

Let $M,r\in (0,+\infty)$ and $S=S_{M,r}$ consist of all $f\in C(L)$ for which there is a holomorphic $F$ on $B(\zeta,r)$ bounded by $M$ and agreeing with $f$ over $B(\zeta,r)\cap L$. We will show that $S$ is closed with empty interior in $C(L)$.

To show that $S$ is closed, let $f_m\in S$ converge to $f$ uniformly in the compact subsets of $L$. There are holomorphic $F_m$ on $B(\zeta,r)$ agreeing with $f_m$ on $B(\zeta,r)\cap L$ and bounded by $M$. By Montel's Theorem, there is a subsequence $F_{k_m}$ and a holomorphic function $F$ on $B(\zeta,r)$ such that $F_{k_m}\to F$ uniformly over all compact subsets of $B(\zeta,r)$. Since $F_{k_m}=f_{k_m}$ over $B(\zeta,r)\cap L$, it follows that $F=f$ over $B(\zeta,r)\cap L$, hence $f\in S$ ($F$ is bounded by $M$ as it is the point-wise limit of functions $F_{k_m}$ bounded by $M$).

Now assume $S$ has nonempty interior, and let $f$ be an interior point. Then there is some $\epsilon>0$ and $N\in \N$, so that if $g\in C(L)$ satisfies
\begin{equation}\label{DistanceOfFunctionsUni}\sup_{z\in L, \|z\|\le N}|f(z)-g(z)|<\epsilon\end{equation}
then $g\in S$. We distinguish two cases, depending on whether or not $B(\zeta,r)\subseteq L$. If $B(\zeta,r)\subseteq L$ then the function $g(z_1,...,z_n)=f(z_1,...,z_n)+\delta \bar z_1$ is in $C(L)$ and satisfies equation \eqref{DistanceOfFunctionsUni} for small enough $\delta>0$, hence it must have a holomorphic extension over $B(\zeta,r)$, contradicting that $(z_1,...,z_n)\mapsto \bar z_1$ is not holomorphic.

If $B(\zeta,r)$ is not a subset of $L$, then after shrinking $r$ we may pick by the first condition on $L$ some $w\in B(\zeta,r)\setminus L$ and an $n-1$ dimensional complex hyperplane $H$ disjoint from $L\cap B(\zeta,r)$ and containing $w$. If $(v_1,...,v_n)\in \C^n$ is a non-zero vector perpendicular to $H$, the function
\begin{equation}\label{QuotientDefined}h(z)=\frac{\delta}{(z_1-w_1)v_1+\cdots+(z_n-w_n)v_n}\end{equation}
is holomorphic on $H^c$, for any $\delta>0$.

 Pick $r'$ with $\|\zeta-w\|<r'<r$; then $h|_{\overline{B(\zeta,r')}\cap L}$ is continuous and by Tietze's Extension Theorem, it has a continuous extension $\tilde h$ on $L$ with $\|\tilde h\|_{\infty,L}\le \|h\|_{\infty, \overline{B(\zeta,r')}\cap L}$. It is then easy to see that the function defined as $h$ on $\overline{B(\zeta,r')}\setminus H$ and $\tilde h$ on $L$ is continuous on $[\overline{B(\zeta,r')}\setminus H]\cup L$, since $[\overline{B(\zeta,r')}\setminus H]\cap L=\overline{B(\zeta,r')}\cap L$ and $\tilde h=h$ over this set. We denote this new extension by $h$; then $h\in C(L)$, $h$ is given by \eqref{QuotientDefined} over $B(\zeta,r')\setminus H$ so it is holomorphic on that set and finally, $\|h\|_{\infty,L}=\|h\|_{\infty,\overline{B(\zeta,r')}\cap L}$. \\
If $g=f+h$ then $g\in C(L)$ and for small enough $\delta$, $\eqref{DistanceOfFunctionsUni}$ is satisfied since \eqref{QuotientDefined} can be used to bound $\|h\|_{\infty,\overline{B(\zeta,r')}\cap L}=\|h\|_{\infty,L}$. We conclude that $g\in S$. If $F,G$ are holomorphic on $B(\zeta,r')$ agreeing with $f,g$ over $B(\zeta,r')\cap L$ respectively, then $G-F=h$ over $B(\zeta,r')\cap L$. Take an open polydisc $U=D_1\times...\times D_n$ centered at $\zeta$ and contained in $B(\zeta,r')\setminus H$. By the second condition on $L$ and Lemma \ref{AnalyticContinuationVersion} below, $G-F=h$ over $U\cap L$ implies that $G-F=h$ on $U$ and thus on $B(\zeta,r')\setminus H$, as the latter set is connected. But because $G,F$ are bounded by $M$, this implies that $h$ is bounded over $B(\zeta,r')\setminus H$, contradicting that $h(z)$ is unbounded as $z\to w$, $z\notin H$.

So far, we have proven that $S_{M,r}$ is closed with empty interior in $C(L)$. The collection of non-extendable functions at $\zeta$ in $C(L)$ can be written as $\cap_{0<M,r<+\infty}S^c_{M,r}$, and after enumerating $M\in \N, r=1/k$, $k\in \N$, this is a countable intersection of dense open sets. Thus by Baire's Theorem, the desired collection of functions is a dense $G_{\delta}$ of $C(L)$. The proof will be complete once we prove the following Lemma \ref{AnalyticContinuationVersion}.\end{proof}

\begin{lem}\label{AnalyticContinuationVersion}Let $U_i$ be domains in $\C$ and $L\subseteq U=U_1\times...\times U_n$ be so that for every $(z_1,...,z_n)\in L$, the sets $\{w\in U_i: (z_1,...,z_{i-1},w,z_i,...,z_n)\in L\}$ have a limit point in $U_i$ for all $i$. If $f$ is holomorphic on $U$ and identically zero on $L$ then it is identically zero on $U$.\end{lem}

\begin{proof}For simplicity take $n=2$. For fixed $(z_0,w_0)\in L$, the function $f(\cdot,w_0)$ is holomorphic on $U_1$ and identically zero on $\{z\in U_1:(z,w_0)\in L\}$, hence by the principle of analytic continuation, $f(z,w_0)=0$ for all $z\in U_1$. Therefore, $f(z,w)=0$ for all $z\in U_1$ and all $w\in U_2$ such that $(z_0,w)\in L$.  For fixed $z\in U_1$, the function $f(z,\cdot)$ is holomorphic on $U_2$ and identically zero on $\{w\in U_2:(z_0,w)\in L\}$ hence again by analytic continuation, it is identically zero on $U_2$. Thus, $f(z,w)=0$ for all $(z,w)\in U$.\end{proof}

We momentarily return to the case of the polydisc. Let $C^p(\T^n)$ denote the space of functions $f:\T^n\to \C$ whose derivatives
\begin{equation*}\frac{\partial ^{\alpha}f}{\partial \theta^{\alpha}}(e^{i\theta_1},...,e^{i\theta_n})=\frac{\partial ^{|\alpha|}f}{\partial \theta_1^{a_1}\partial \theta_2^{a_2}\cdots \partial \theta_n^{a_n}}(e^{i\theta_1},...,e^{i\theta_n})\end{equation*}
exist and are continuous on $\T^n$ for all $\alpha=(a_1,...,a_n)\in \N^n$ with $|\alpha|=a_1+\cdots+a_n\le p$. The space $C^p(\T^n)$, which is topologized by the seminorms
\begin{equation}|f|_{\alpha}=\sup_{\theta_i\in \R}\left|\frac{\partial ^{\alpha}f}{\partial \theta^{\alpha}}(e^{i\theta_1},...,e^{i\theta_n})\right|\text{ , }|\alpha|\le p\end{equation}
is a Banach space for $p<+\infty$ and a Fr\'echet space for $p=+\infty$. 

If $f\in A^p(D^n)$ then its restriction $f|_{\T^n}$ is in $C^p(\T^n)$, as one can easily see. To show that the restriction map is an embedding of Fr\'echet spaces, one has to note that the usual topology on $A^p(D^n)$ given by the seminorms
\begin{equation}|f|_{\alpha}=\sup_{z\in D^n}\left|\frac{\partial ^{\alpha}f}{\partial z^{\alpha}}(z)\right|=\sup_{\theta_i\in \R}\left|\frac{\partial ^{\alpha}f}{\partial z^{\alpha}}(e^{i\theta_1},...,e^{i\theta_n})\right|\end{equation}
(the last equality follows from the maximum modulus Theorem) is also induced by the seminorms
\begin{equation}|f|_{\alpha,*}=\sup_{z\in D^n}\left|\frac{\partial ^{\alpha}f}{\partial \theta^{\alpha}}(z)\right|\end{equation}
To see this, note that the derivative with respect to $\theta^{\alpha}$ is a complex polynomial of the derivatives with respect to $z^{\beta}$, $|\beta|\le |\alpha|$, $z_i=e^{i\theta_i}$. This can be proven via induction on $|\alpha|$. The base case $|\alpha|=1$ is the following:
\begin{equation}\frac{\partial f}{\partial \theta_i}=\frac{\partial f}{\partial e^{i\theta_i}}\frac{\partial e^{i\theta_i}}{\partial \theta_i}=\frac{\partial f}{\partial z_i}ie^{i\theta_i}\end{equation}
The space $A^p(D^n)$ is thus embedded canonically in $C^p(\T^n)$, via the restriction map.

\begin{thm}\label{CpTn}For $p\in \{0,1,...\}\cup \{+\infty\}$, the collection of functions in $C^p(\T^n)$ that are not real analytic at a fixed point of $\T^n$ (or at all points of $\T^n$) is a dense $G_{\delta}$ subset of $C^p(\T^n)$.\end{thm}

\begin{proof}Real analyticity is equivalent to holomorphic extendability by Proposition \ref{RealAnalyticIsHolomorphicallyExtendable}. So the case $p=0$ follows directly from Theorem \ref{TheoremWithClosedSet}. The cases $p=1,...,+\infty$ can be proven similarly to the proof of Theorem \ref{TheoremWithClosedSet}. Equation \eqref{DistanceOfFunctionsUni} has to be changed to
\begin{equation}\label{CpTnequation}\sup_{z\in \T^n}\left|\frac{\partial^{\alpha}f}{\partial \theta^{\alpha}}(z)-\frac{\partial^{\alpha}g}{\partial \theta^{\alpha}}(z)\right|<\epsilon\text{ , }\forall \alpha\in \N^n\text{ , }|\alpha|\le l\end{equation}
where $l\in \N$, $l\le p$, depends only on $\epsilon$ (and $f,\zeta$ of course). Note that the complex hyperplane can always be chosen to have an orthogonal coordinate vector: if $(w_1,...,w_n)\notin \T^n$ then some $w_i$ is not in $\T$, and we can use the hyperplane $z_i=w_i$ that is orthogonal to the vector $e_i=(0,...,0,1,0,...,0)$, the unit being in the $i$-th position. Thus, the function $h(z)=\delta (z_i-w_i)^{-1}$ is in $C^p(\T^n)$ (so the argument with Tietze's Theorem need not be used), and for small enough $\delta>0$, it is easy to see that \eqref{CpTnequation} is verified for $g=f+h$. 
 \end{proof}

\begin{thm}\begin{itemize}
\item[1.] The collection of functions in $C(\partial D^n)$ that are not holomorphically extendable at a fixed point of $\partial D^n$ (or at all points of $\partial D^n$) is a dense $G_{\delta}$ of $C(\partial D^n)$.\smallbreak
\item[2.] The collection of functions in $C(\partial B^n)$ that are not holomorphically extendable at a fixed point of $\partial B^n$ (or at all points of $\partial B^n$) is a dense $G_{\delta}$ of $C(\partial B^n)$.
\end{itemize}
\end{thm}
\begin{proof}These follow from Theorem \ref{TheoremWithClosedSet}. Note that by picking the point $w$ in that Theorem to be outside of $D^n$ or $B^n$, we can have our hyperplane $H$ be disjoint from the boundaries $\partial D^n$ and $\partial B^n$ respectively. \end{proof}

Now let $\Omega_1,...,\Omega_n\subseteq \C$ be Jordan domains in the complex plane and let $\phi_i:D\to \Omega_i$ denote the Riemann mappings. By the Caratheodory-Osgood Theorem, the maps $\phi_i$ extend to homeomorphisms $\phi_i:\overline D\to \overline{\Omega_i}$. Let $\gamma_i=\phi_i|_{\T}$ and consider $\Omega=\Omega_1\times\cdots\times \Omega_1$, $\gamma=\gamma_1\times\cdots\times \gamma_n:\T^n\to \partial \Omega_1\times\cdots \times\partial \Omega_n$ and  $\phi=\phi_1\times\cdots \times \phi_n:\overline D^n\to \overline{\Omega}$. Obviously, $\gamma=\phi|_{\T^n}$.

The space $C^p(\gamma)$ consists of all functions $f:\gamma(\T^n)\to \C$ such that $f\circ \gamma\in C^p(\T^n)$. Its topology is defined by the seminorms 
\begin{equation}|f|_{\alpha}=\sup_{\theta_i\in \R}\left|\frac{\partial ^{\alpha}(f\circ \gamma)}{\partial \theta^{\alpha}}(e^{i\theta_1},...,e^{i\theta_n})\right|\end{equation}
where $\alpha\in \N^n$ is a multi-index with $|\alpha|\le p$. It is easy to see that $C^p(\gamma)$ and $C^p(\T^n)$ are isometrically isomorphic via the map $f\mapsto f\circ \gamma$.

If $\gamma_i\in C^p(\T)$ then $\phi_i\in A^p(D)$ and the converse is also true (\cite{Vlasis}). 

If $\gamma_i$ is an analytic curve (i.e. it extends to a biholomorphism between an open neighborhood of $\T$ and one of $\partial \Omega_i$)  then $\phi_i$ extends to a conformal equivalence between a larger disc $D(0,r)$, $r>1$, and an open neighborhood of $\overline{\Omega_i}$. This is a consequence of the Schwarz reflection principle (\cite{Alhfors}).  Therefore, if all $\gamma_i$ are analytic curves, then $\phi$ extends to a conformal equivalence between an open neighborhood of $\overline D^n$ and an open neighborhood of $\overline \Omega$.\medbreak

A function $f:\partial \Omega_1\times\cdots \partial \Omega_n\to \C$ is real analytic at a point $\gamma(t)$, $t\in \T^n$, if the composition $f\circ \gamma$ is real analytic at $t\in T^n$.

A function $f:\partial \Omega_1\times\cdots \partial \Omega_n\to \C$ is holomorphically extendable at a point in $\gamma(\T^n)$  if there is an open neighborhood $U$ of that point and a holomorphic function $F:U\to \C$ so that $F=f$ on $U\cap \T^n$.

\begin{thm}\label{RealAnalyticIsHolomorphicallyExtendableNo2}If all $\gamma_i$ are analytic curves, then $f$ is real analytic at a point in $\gamma(\T^n)$ if and only if it is holomorphically extendable at that point.\end{thm}

\begin{proof}This follows from Proposition \ref{RealAnalyticIsHolomorphicallyExtendable}.\end{proof}

Theorems \ref{CpTn} and \ref{RealAnalyticIsHolomorphicallyExtendableNo2} imply that:

\begin{thm}If $p\in \{0,1,...\}\cup \{+\infty\}$ and all $\gamma_i$ are analytic curves, then the collection of functions in $C^p(\gamma)$ that are not real analytic at a fixed point (or at all points) of $\gamma(\T^n)$ is a dense $G_{\delta}$ of $C^p(\gamma)$\end{thm}

Note that the definition of holomorphic extendability also makes sense for continuous curves $\gamma_i:\T\to \C$, with $\gamma=\gamma_1\times...\times \gamma_n:\T^n\to \C^n$. We then have the following:

\begin{thm}If $p\in \{0,1,...\}\cup \{+\infty\}$ and all $\gamma_i$ are locally injective $C^p$ curves, then the collection of functions in $C^p(\gamma)$ that are not holomorphically extendable at a fixed point (or at all points) of $\gamma(\T^n)$ is a dense $G_{\delta}$ of $C^p(\gamma)$\end{thm}

\begin{proof}The proof is similar to that of Theorem \ref{CpTn}. The assumption $\gamma\in C^p(\T)$ is needed to verify that the function $h$ defined in \eqref{QuotientDefined} is in $C^p(\gamma)$. The first condition of Theorem \ref{TheoremWithClosedSet} is satisfied since if $(w_1,...,w_n)\notin \gamma(\T^n)=\gamma_1(\T)\times\cdots\times \gamma_n(\T)$ then some $w_i$ is not $\gamma_i(\T)$, hence the complex hyperplane $z_i=w_i$ does not intersect $\gamma(\T^n)$. Local injectivity suffices for the second condition of Theorem \ref{TheoremWithClosedSet}, since $\gamma$ is locally a homeomorphism onto its image.\end{proof}

\section{One-sided extendability}\label{One-Sided Extendability}

We first consider different notions of one-sided holomorphic extendability for functions defined on $\T^n$. Because $\C^n\setminus \T^n$ is connected, the term ``one-sided'' is an abuse of terminology and the two sides are those given by the polydisc $D^n\subseteq \C^n$ and the complement ${\overline {D^n}}^c$.

A function $f:\T^n\to \C$ is holomorphically extendable from one side at a point $z\in \T^n$, if there is an open neighborhood $U$ of $z$ and a continuous function $F:[U\cap D^n]\cup \T^n\to \C$ that is holomorphic on $U\cap D^n$ and agrees with $f$ on $\T^n$. A slightly different definition would be to have $F$ defined and continuous on $\overline{U\cap D^n}$; after shrinking $U$, this essentially amounts to $F$ being continuous on $[U\cap D^n]\cup \partial D^n$. In the following discussion and Theorems, both two variants behave in the same way, so we shall consider them equivalent.

For technical reasons, it is useful to relax the continuity of $F$ to boundedness and separate continuity in its variables. We may also strengthen continuity of $F$ to Lipschitz continuity of $F$. Let us call these two notions of one-sided extendability as ``separately continuous one-sided holomorphic extendability'' and ``Lipschitz continuous one-sided holomorphic extendability'' respectively. 

Finally let us fix  $p\in \{0,1,...\}\cup \{+\infty\}$.

\begin{thm}\label{OneSidedTorus}\begin{itemize}\item[1.] The collection of functions in $C^p(\T^n)$ that are not holomorphically extendable from one side in the separate continuous sense at all points of $\T^n$ is a dense $G_{\delta}$ of $C^p(\T^n)$.   \smallbreak
\item[2.] The collection of functions in $C^p(\T^n)$ that are not holomorphically extendable from one side (in the usual sense) at all points of $\T^n$ is residual in $C^p(\T^n)$, i.e. it contains a dense $G_{\delta}$ subset.   \smallbreak
\item[3.] The collection of functions in $C^p(\T^n)$ that are not holomorphically extendable from one side in the Lipschitz continuous sense at all points of $\T^n$ is a dense $G_{\delta}$ of $C^p(\T^n)$.\end{itemize}\end{thm}

\begin{proof}Obviously, the collection in item 2. contains that of item 1. and is contained in that of item 3. So it suffices to prove items 1. and 3.

Let us first show item 1. Fix $\zeta=(\zeta_1,...,\zeta_n)\in \T^n$, $M,r\in (0,+\infty)$, $U=D(\zeta_1,r)\times\cdots \times D(\zeta_n,r)$ and $V=D^n\cap U$. Let $S=S_{M,r}$ consist of all $f\in C^p(\T^n)$ for which there is a function $F:V\cup \T^n\to \C$ so that
\begin{itemize}\item $F$ is bounded by $M$ and is separately continuous in its variables
\item $F$ is holomorphic on $V$
\item $F=f$ over $\T^n$\end{itemize}
We will show that $S$ is closed with empty interior in $C^p(\T^n)$. The same argument works if we require $F$ to be defined on $V\cup \partial D^n$ and not just $V\cup \T^n$, and have the previous three properties.

To show that $S$ is closed, let $f_m\in S$ converge to $f$ in the topology of $C^p(\T^n)$. There are functions $F_m$ with the previous three properties (but this time $F_m=f_m$ over $\T^n$). By Montel's Theorem, there is a subsequence $F_{k_m}$, that we take to be $F_m$ for simplicity, converging to some function $F$ uniformly on the compact subsets of $V$. The function $F$ is holomorphic on $V$, and we extend it over $\T^n$ by setting it equal to $f$. Since $F_m\to F$ point-wise over $V\cup \T^n$, $F$ is bounded by $M$, so it remains to show that it is separately continuous. If we fix $(z_1,...,z_n)\in V\cup \T^n$ then the functions $R_m(z)=F_m(z,z_2,...,z_n)$ are continuous on $[D\cap D(\zeta_1,r)]\cup \T$ and holomorphic on $D\cap D(\zeta_1,r)$. The function $R(z)=F(z,z_2,...,z_n)$ is also holomorphic on $D\cap D(\zeta_1,r)$ and $R_m\to R$ uniformly on the compact subsets of $[D\cap D(\zeta_1,r)]$ and uniformly on $\T$. By the proof of Proposition 6.4. of \cite{BigPaper}, which is based on the Poisson representation, we conclude that $R$ is continuous on $[D\cap D(\zeta_1,r)]\cup \T$, hence that $F$ is continuous in its first variable. We can similarly prove continuity with respect to its other variables. Therefore, $F$ has all three required conditions hence $f\in S$, which proves that $S$ is closed.

Let us assume that $S$ has an interior point $f$. Then there is some $\epsilon>0$ and $l\in \N$, $l\le p$, so that if $g\in C^p(\T^n)$ and
\begin{equation}\sup_{z\in \T^n}\left|\frac{\partial^{\alpha}f}{\partial \theta^{\alpha}}(z)-\frac{\partial^{\alpha}g}{\partial \theta^{\alpha}}(z)\right|<\epsilon\text{ , }\forall \alpha\in \N^n\text{ , }|\alpha|\le l\end{equation}
then $g\in S$. 
Pick $(w_1,...,w_n)\in U$ with $|w_1|<1$. If $h(z_1,...,z_n)=\delta(z_1-w_1)^{-1}$ for small $\delta>0$, then $g=f+h$ satisfies the inequality above, hence $g\in S$. If $F,G$ denote the separately continuous extensions of $F,G$ over $V\cup \T^n$, that are holomorphic on $V$ and bounded by $M$, then $G=F+h$ on $\T^n$. The functions $G_1(z)=G(z,z_2,...,z_n)$ and $F_1(z)=F(z,z_2,...,z_n)$ are continuous on $(D(\zeta_1,r)\cap D)\cup A$, where $A$ is an open arc contained in $\overline{D(\zeta_1,r)\cap D}\cap \T$, are holomorphic on $D(\zeta_1,r)\cap D$ and $G_1=F_1+\delta(z-w_1)^{-1}$ on $\T$. Thus, the function $R=G_1-F_1-\delta(z_1-w_1)^{-1}$ is continuous on $(D(\zeta_1,r)\cap D \setminus \{w_1\})\cup A$, holomorphic on $(D(\zeta_1,r)\cap D \setminus \{w_1\})$ and vanishes on $A$. By the Schwarz Reflection Principle, $R$ can be holomorphically extended over an open neighborhood of $A$, hence the principle of analytic continuation implies that $R=0$ on $D(\zeta_1,r)\cap D \setminus \{w_1\}$. Therefore, $G_1-F_1=\delta(z-w_1)^{-1}$ on $D(\zeta_1,r)\cap D\setminus \{w_1\}$, which is a contradiction since $G_1-F_1$ is bounded by $M$ as $z\to w_1$ while $\delta(z-w_1)^{-1}$ is not.

We have proven thus far that $S=S_{M,r}$ is closed with empty interior. By Baire's Theorem, $A_{\zeta}=\cap_{M,k\in \N_+}S_{M,1/k}$ is a dense $G_{\delta}$ subset of $C^p(\T^n)$, and if $\zeta_m\in \T^n$ is a dense sequence in $\T^n$ then $\cap_mA_{\zeta_m}$ is also a dense $G_{\delta}$. This is the collection described in item 1.

Proving item 3. is similar with the following differences: First, in the definition of $S$ we now require $F$ to be Lipschitz continuous and its Lipschitz constant to be less than $M$. Second, the proof of $S$ being closed is modified as follows: We have $f_m\to f$ in the topology of $C^p(\T^n)$, and Lipschitz continuous functions $F_m:V\cup \T^n\to \C$, holomorphic on $V$, bounded by $M$, with Lipschitz constants less than $M$, and with $F_m=f_m$ over $\T^n$.  The family of the $F_m$ is equicontinuous and bounded, hence by the Arzela-Ascoli Theorem, there is a subsequence  $F_{k_m}$ converging to some continuous function $F$, uniformly on $V\cup\T^n$. The function $F$ clearly has the properties all $F_m$ share, hence $f\in S$. The rest of the proof of item 3. is entirely analogous to that of item 1. \end{proof}

Note that the one-sided extensions we have considered are in the inner side of the unit polydisc. The same result is true for extensions from the outer side of $\partial D^n$, that is $\overline{D^n}^c$, as can be shown by composing by the biholomorphism $(z_1,...,z_n)\mapsto (1/z_1,...,1/z_n)$ (when $(z_1,...,z_n)$ is sufficiently close to a point in $\T^n$ then all $z_i\neq 0$).\medbreak

Now let $\Omega_i\subseteq \C$ be Jordan domains; we denote by $\phi_i$ the Riemann maps $D\to \Omega_i$ that extend to homeomorphisms $\gamma_i:\T\to \partial \Omega_i$ by the Caratheodory-Osgood Theorem. We set $\gamma=\gamma_1\times...\times \gamma_n$, $\phi=\phi_1\times...\times \phi_n$ and $\Omega=\Omega_1\times...\times \Omega_n$, as usual.

 A function $f\in C(\gamma)$ is holomorphically extendable from one side at a point $\zeta\in \gamma(\T^n)$ if there is a radius $r>0$ and a continuous function $F:[B(\zeta,r)\cap \Omega]\cup \gamma(\T^n)\to \C$ that is holomorphic on $B(\zeta,r)\cap \Omega$ and that agrees with $f$ on $\gamma(\T^n)$. 

As in the case of $\T^n$, we still have two additional notions of one-sided extendability. In the weaker one, we require $F$ to be bounded and separately continuous. In the stronger one, we have $F$ be Lipschitz continuous.
\begin{prop}\label{OneSidedAnalyticProp} If all $\gamma_i$ are $C^p$ smooth, then $f\in C^p(\gamma)$ is holomorphically extendable from one side at $\gamma(t_0)$ if and only if the function $f\circ \gamma\in C^p(\T^n)$ is holomorphically extendable from one side at $t_0\in \T^n$.\end{prop}

\begin{proof}This easily follows from the definitions using the fact that $\phi\in A^p(D^n)$ if for all $i=1,...,n$ we have $\gamma_i\in C^p(\T)$.\end{proof}

Proposition \ref{OneSidedAnalyticProp}  holds for the notion of separately continuous one-sided extendability as well. For the notion of Lipschitz extendability, we should also add the condition that $\phi,\phi^{-1}$ be Lipschitz continuous.

\begin{thm}Let $\Omega,\gamma$ be as before and $\gamma$ be $C^p$ smooth.
\begin{itemize}\item[1.]The set of functions $C^p(\gamma)$ that are not holomorphically extendable from one side in the separate continuous sense at all points of $\gamma(\T^n)$, is a dense $G_{\delta}$ of $C^p(\gamma)$.
\smallbreak
\item[2.] The set of functions in $C^p(\gamma)$ that are not holomorphically extendable from one side (in the usual sense) at all points of  $\gamma(\T^n)$ is residual in $C^p(\gamma)$, i.e. it contains a dense $G_{\delta}$.   
\smallbreak
\item[3.] If $\phi,\phi^{-1}$ are Lipschitz continuous, the set of functions in $C^p(\gamma)$ that are not holomorphically extendable from one side in the Lipschitz continuous sense at all points of $\gamma(\T^n)$, is a dense $G_{\delta}$ in $C^p(\gamma)$.\end{itemize}\end{thm}

\begin{proof}This follows from Theorem \ref{OneSidedTorus} and Proposition \ref{OneSidedAnalyticProp}.\end{proof}

If all $\gamma_i$ are assumed to be analytic curves, then we can similarly prove an analogous Theorem for holomorphic extensions from the outer side of $\partial \Omega$, that is extensions in $\overline{\Omega}^c$.\medbreak

Another natural notion of one-sided holomorphic extendability concerns functions defined on the entire boundary $\partial D^n$ of the polydisc. 
A function $f:\partial D^n\to \C$ is holomorphically extendable from one side at $\zeta\in \partial D^n$, if there is an open neighborhood $U$ of $\zeta$ and a continuous function $F:[U\cap D^n]\cup \partial D^n\to \C$ that is holomorphic on $U\cap D^n$ and agrees with $f$ on $\partial D^n$. As usual, we have the two associated notions of separately continuous and Lipschitz continuous extensions.

\begin{thm}\label{OneSidedPolydiscFullBoundary}
\begin{itemize}\item[1.]The set of functions $C(\partial D^n)$ that are not holomorphically extendable from one side in the separate continuous sense at all points of $\partial D^n$, is a dense $G_{\delta}$ of $C(\partial D^n)$.\smallbreak
\item[2.] The set of functions in $C(\partial D^n)$ that are not holomorphically extendable from one side (in the usual sense) at all points of $\partial D^n$ is residual in $C(\partial D^n)$.\smallbreak
\item[3.] The set of functions in $C(\partial D^n)$ that are not holomorphically extendable from one side in the Lipschitz continuous sense at all points of  $\partial D^n$, is a dense $G_{\delta}$ in $C(\partial D^n)$.\end{itemize}\end{thm}

\begin{proof}The proof is similar to that of Theorem \ref{OneSidedTorus}, the difference being in how the set $S$ is shown to have empty interior. We use the same notation as in the proof of that Theorem, and fix $\zeta=(\zeta_1,...,\zeta_n)\in \partial D^n$. For a single $i$, $|\zeta_i|=1$; say $|\zeta_1|=1$ for simplicity, and let $0<r<1-|\zeta_i|$ for all $i\neq 1$ and $U=D(\zeta_1,r)\times\cdots\times D(\zeta_n,r)$. Pick $w\in D^n\cap U$ with $|w_1|<1$ and let the hyperplane $H$ be given by $z_1=w_1$. The set $H$ is disjoint from $U\cap \partial D^n$ since if $z=(z_1,...,z_n)\in U\cap H$ then $|z_1|=|w_1|<1$ and for $i\neq 1$, $|z_i-\zeta_i|<r\implies |z_i|<|\zeta_i|+r<1$, hence $z\in D^n$ and is not in $\partial D^n$. 

Following the argument in the proof Theorem \ref{TheoremWithClosedSet} involving the Tietze Extension Theorem, we can prove that the function $(z_1,...,z_n)\mapsto \delta/(z_1-w_1)$ defined on $U\setminus H$ has a continuous extension on $\partial D^n$. The proof is then completed using the reflexion principle in each variable separately, as in the proof of Theorem \ref{OneSidedTorus}.\end{proof}

Note that the same result holds for extensions on the outer side of the polydisc. The proof is even simpler, as the point $w$ is taken in $(\overline{D^n})^c$, with $|w_1|>1$, and hence the hyperplane $z_1=w_1$ does not intersect $\partial D^n$. So the Tietze Extension Theorem argument need not be used in this case.\medbreak

We end this section with the case of the sphere: A function $f:\partial B^n\to \C$ is holomorphically extendable from one side at a point $\zeta\in \partial B^n$, if there is a an open neighborhood $U$ of  $\zeta$ and a continuous function $F:[U\cap (B^n)^c]\cup \partial B^n\to \C$ that is holomorphic on $U\cap (B^n)^c$ and agrees with $f$ on $\partial B^n$. As always, we have the two associated notions of separately continuous and Lipschitz continuous extensions.

\begin{thm}\label{OneSidedSphere}
\begin{itemize}\item[1.]The set of functions $C(\partial B^n)$ that are not holomorphically extendable from one side in the separate continuous sense at all points of $\partial B^n$, is a dense $G_{\delta}$ of $C(\partial B^n)$.\smallbreak
\item[2.] The set of functions in $C(\partial B^n)$ that are not holomorphically extendable from one side (in the usual sense) at all points of $\partial B^n$ is residual in $C(\partial B^n)$.\smallbreak
\item[3.] The set of functions in $C(\partial B^n)$ that are not holomorphically extendable from one side in the Lipschitz continuous sense at all points of  $\partial B^n$, is a dense $G_{\delta}$ in $C(\partial B^n)$.\end{itemize}\end{thm}

\begin{proof}The proof is analogous to that of Theorem \ref{OneSidedPolydiscFullBoundary}. To show the existence of the desired hyperplane $H$, note that due to the homogeneity of the sphere, we can take $\zeta=(1,0,...,0)$. If $w=(1-r,0,....,0)$ for very small $r>0$, then the hyperplane defined by $z_1=1-r$ passes through $w$ but does not intersect $\partial B^n$ at any point of $B(\zeta,\sqrt{3r})$. The rest of the proof is entirely similar. \end{proof}

Just as in the case of $\partial D^n$, the same result holds for extensions from the outer side of $\partial B^n$, that is extensions in $(\overline{B^n})^c$.\medbreak

\noindent \textbf{Acknowledgment}: We would like to thank Professor V. Nestoridis for bringing these problems to our attention, and for his always useful suggestions and help. We would also like to thank Professor T. Hatziafratis for taking an interest in this paper.

\end{document}